\documentclass[]{article}

\usepackage{
	subfiles,
	amsmath,		
	amssymb,		
	amsthm,	
	fancyhdr,	
	gensymb,		
	wrapfig, 
	tasks, 
	todonotes, 
	mathtools,
	faktor, 
	xcolor,
	soul,
	paracol,
	url,
	listings,
	MnSymbol}
\usepackage[noblocks]{authblk}
\usepackage[font=small,labelsep=none]{caption}
\usepackage{listings}
\usepackage[hidelinks]{hyperref}
\usepackage[noabbrev]{cleveref}
\usepackage{tikz-cd}
\usepackage{tikz}
\usepackage[english]{babel}
\usepackage[inline]{enumitem}

\usepackage[margin=3cm, a4paper]{geometry}

\newtheorem{theorem}{Theorem}[section]

\newtheorem{lemma}[theorem]{Lemma}
\newtheorem{proposition}[theorem]{Proposition}
\newtheorem{claim}{Claim}

\newtheoremstyle{remarkk}% <name>
{3pt}% <Space above>
{3pt}% <Space below>
{}% <Body font>
{}% <Indent amount>
{\itshape \bfseries}% <Theorem head font>
{.}% <Punctuation after theorem head>
{.5em}% <Space after theorem headi>
{}% <Theorem head spec (can be left empty, meaning `normal')>
\theoremstyle{remarkk}
\newtheorem{examplex}[theorem]{Example}
\newenvironment{example}
{\pushQED{\qed}\examplex}
{\popQED\endexamplex}
\newtheorem{remarkx}[theorem]{Remark}
\newenvironment{remark}
{\pushQED{\qed}\remarkx}
{\popQED\endremarkx}

\theoremstyle{definition}
\newtheorem{definition}[theorem]{Definition}

\numberwithin{equation}{section}

\newcommand{\Z}{\mathbb{Z}}

\newcommand{\C}{\mathbb{C}}
\newcommand{\F}{\mathbb{F}}
\newcommand{\p}{\mathfrak{p}}

\newcommand{\m}{\mathfrak{m}}
\renewcommand{\a}{\mathfrak{a}}
\newcommand{\q}{\mathfrak{q}}

\newcommand{\cc}{\mathfrak{c}}
\renewcommand{\O}{\mathcal{O}}

\renewcommand{\phi}{\varphi}

\renewcommand{\epsilon}{\varepsilon}

\newcommand{\hil}{\textup{[}}
\newcommand{\hr}{\textup{]}}

\usepackage[OT2,T1]{fontenc}
\DeclareSymbolFont{cyrletters}{OT2}{wncyr}{m}{n}
\DeclareMathSymbol{\Sha}{\mathalpha}{cyrletters}{"58}

\DeclareMathOperator{\sgn}{sgn}

\DeclareMathOperator{\rank}{rank}
\newcommand{\Leg}[2]{\left( \frac{#1}{#2}\right)}

\DeclareMathOperator{\Gal}{Gal}
\DeclareMathOperator{\ord}{ord}
\DeclareMathOperator{\id}{id}

\newcommand{\Art}[3]{ \text{Art}(#1,#2/#3) }
\renewcommand{\hat}[1]{\widehat{#1}}

\makeatletter
\renewcommand{\mod}[1]{{\@displayfalse\pmod{#1}}}
\makeatother

\title{Governing fields for hyperelliptic function fields}

\author{Joppe Stokvis \thanks{Mathematical Institute, Leiden University and QuSoft, Amsterdam, \href{mailto: j.a.stokvis@math.leidenuniv.nl}{j.a.stokvis@math.leidenuniv.nl}}}

\begin{document}

\maketitle

\begin{abstract}
	We study the $8$-rank of class groups of hyperelliptic function fields and show that such $8$-ranks are governed by splitting conditions in so-called governing fields. A similar result was proven for quadratic number fields by Stevenhagen, who used a theory of R\'edei symbols and R\'edei reciprocity to do so. We introduce a version of the R\'edei reciprocity law for function fields and use this to show existence of governing fields.
\end{abstract}

\section{Introduction}
Since Gauss introduced his genus theory \cite{Gauss}, the $2$-part of (narrow) class groups of quadratic number fields has been extensively studied. By the Cohen-Lenstra heuristics \cite{Cohen-Lenstra} we expect only this $2$-part to have a predictable structure. To understand the structure of the $2$-part, Cohn and Lagarias introduced governing fields and proved their existence for the $2$- and $4$-rank \cite{Cohn-Lagarias}. Stevenhagen gave a proof for the existence of governing fields for the $8$-rank \cite{Stevenhagen2} using class field theory. More recently, a new proof of their existence \cite{Stevenhagen1} was given by using the R\'edei symbol, which satisfies the R\'edei reciprocity law. This trilinear symbol is named after the Hungarian mathematician who was the first to give a definition \cite{Redei2}. On the existence of a governing field of the $16$-rank for quadratic number fields there is only a negative answer by Koymans and Milovic \cite{Koymans-Milovic}, under the assumption of a short character sum conjecture. 

Similarly to their number field equivalent, class groups of function fields over finite fields show the same structural behaviour in their $2$-part, as was first discovered by Artin in his thesis \cite{Artin}. Because of heuristics by Friedman and Washington \cite{Friedman-Washington} that resemble those of Cohen-Lenstra, this structural behaviour is expected to occur only in the 2-part. This motivates the introduction of governing fields for hyperelliptic function fields, which will be done in this article.

Let $q$ be an odd prime power and consider the rational function field $\F_q(x)$ over the finite field with $q$ elements. A hyperelliptic function field is a quadratic extension $K_D = \F_q(x, \sqrt{D})$ determined by a squarefree polynomial $D \in \F_q[x]$. Its ring of integers $\O_D$ is the algebraic closure of $\F_q[x]$ in $K_D$ and we define $C(D)$ (also denoted by $C$) as the ideal class group of $\O_D$.  Even though we will write $C(D)$ as a multiplicative group, we can make an identification to a vector space over $\F_2$. For $j \in \Z_{\geq1}$ the $2^j$-rank of $C(D)$ is defined as the dimension of the $\F_2$-vector space $C^{2^{j-1}}/C^{2^j}$.
\begin{equation}\label{def: 2^j-rank}
	r_{2^j} := \dim_{\F_2}\left(C^{2^{j-1}}/C^{2^j}\right) = \dim_{\F_2}\left(C[2^j]/C[2^{j-1}]\right).
\end{equation}

\begin{definition}[Governing fields]\label{def: gov field intro}
	Let $q$ be an odd prime power and $D\in \F_q[x]$ squarefree. A Galois extension $\Omega_{2^j}(D)/\F_q(x)$ is called a \emph{governing field for the $2^j$-rank of $C(DP)$} if the Frobenius conjugacy class of an unramified prime $P$ in $\Gal(\Omega_{2^j}(D)/\F_q(x))$ completely determines the $2^k$-rank for all $k\leq j$.
\end{definition}

In order to obtain results about governing fields we will need explicit methods to determine the $2^j$-rank of the class group for various $j$. Even though this text treats the calculations as a purely algebraic problem for function fields, we want to highlight that similar results can also be obtained through a geometrical interpretation. By relating the class group to the Jacobian of a hyperelliptic curve one obtains additional geometric structure to study the problem \cite{Cornelissen2, Cornelissen3}.

\subsection{Main results}

This article studies the $2^j$-rank of  hyperelliptic function fields for $j=1,2,3$ and proves the following (see also \cref{thm: gov field 8rank ffields}).
\begin{theorem}\label{thm: omega exists intro}
	Let $D \in \F_q[x]$ be a squarefree polynomial. Then $\Omega_2(D)$ and $\Omega_4(D)$ exist. If not all irreducible factors of $D$ have even degree, then $\Omega_8(D)$ exists as well.
\end{theorem}

The main part of the proof of \cref{thm: omega exists intro} uses a symmetry law of the R\'edei symbol $[A,B,C]$ that takes its arguments in $\F_q(x)^*/\F_q(x)^{*^2}$. Adapting the methods for number fields \cite{Stevenhagen1}, we define the R\'edei symbol for function fields in \cref{def: Redei symbol}. Its entries have to satisfy
	\begin{align}
	&\gcd(A,B,C)=1, \label{intro: redei cond 1} \text{ where not all have odd degree, and }\\
	&(A, B)_P = (A,C)_P = (B, C)_P = 1 \text{ for all primes $P$,} \label{intro: redei cond 2}
\end{align}
where $(A,B)_P$ is the Hilbert symbol at a prime $P$ \cite[chapter XII]{Artin-Tate}.
The R\'edei reciprocity law can then be stated as follows (see also \cref{thm: Redei reciprocity}).
\begin{theorem}\label{thm: reciprocity intro}
	Let $A,B,C \in \F_q(x)^*/\F_q(x)^{*^2}$ satisfy \cref{intro: redei cond 1} and \cref{intro: redei cond 2}. Then 
	\begin{equation}
		[A,B,C] = [B,A,C] = [A,B,C]
	\end{equation} 
	and the R\'edei symbol is multiplicative in all its entries.
\end{theorem}

\subsection{Outline}
In \cref{sec: preliminaries} we introduce some preliminary concepts and explain the general strategy on how to calculate the $2^j$-rank of a class group. This is made explicit for $j=1$ and $j=2$ in \cref{sec: genus and redei} using genus theory \cite{Artin, Semirat, Zhang} and methods by R\'edei \cite{Redei1, Redei2, Wittmann} respectively. The calculation of the $8$-rank is explained in \cref{sec: 8-rank}. In \cref{sec: redei reciprocity} we give a definition of R\'edei symbols and prove the theorem of R\'edei reciprocity. The definition is similar to the one for number fields \cite{Stevenhagen1}, but some additional care has to be taken when defining minimally ramified extensions. R\'edei reciprocity will be the main tool in proving the existence of governing fields, which is done in \cref{sec: governing fields}.

The theory of R\'edei symbols, the reciprocity law and governing fields was originally developed for quadratic number fields. Part of the definitions and methods in this article are based on \cite{Stevenhagen1} and adapted for function fields. The interested reader can compare the number field and function field results using the following conversion of sections:
\[\begin{tabular}{c|c}
	This article &\cite{Stevenhagen1}\\
	\hline
	\cref{sec: genus and redei}&sections 2, 3\\
	\cref{sec: 8-rank}&sections 4, 5\\
	\cref{sec: redei reciprocity} & sections 7, 8\\
	\cref{sec: governing fields} & section 9
\end{tabular}\]

\paragraph{Notation}\label{paragraph: notation}
We always let $q$ be an odd prime power and $k=\F_q(x)$ be the rational function field. Generally, $K_A = k(\sqrt{A})$ is used to denote the quadratic extension for any non-square $A \in \F_q[x]$. In \cref{sec: redei reciprocity} we frequently use the extension $k(\sqrt{AB})$ with $A,B\in \F_q[x]$. We denote it by $K =k(\sqrt{AB})$ and use $E = K(\sqrt{A}) = K(\sqrt{B})=k(\sqrt{A}, \sqrt{B})$ for the quartic extension. A cyclic quartic extension of $K$ is usually denoted by $F$.

The $2^j$-rank of a class group is defined by identifying $C^{2^j}/C^{2^{j-1}}$ as an $\F_2$-vector space. In this article we will always write the class group multiplicatively. Quadratic Galois groups will be represented by $\{\pm1\}$, consistent with notation for quadratic symbols. Whenever the dimension of a vector space needs to be computed, it is assumed that an identification with $\F_2$ is made.

\section{Preliminaries}\label{sec: preliminaries}
Let $q$ be an odd prime and denote by $k = \F_q(x)$ the rational function field. By a \emph{hyperelliptic function field} we mean a quadratic extension $K_D = k(\sqrt{D})$ with $D \in \F_q[x]$ squarefree. The polynomial $D$ has a decomposition into irreducible parts as $D = eP_1\dots P_s$, where the $P_i\in \F_q[x]$ are distinct monic irreducible polynomials and $e \in \F_q^*$ is the sign of $D$. We may assume without loss of generality that $e=1$ or $g$ for some fixed generator $g$ of $\F_q^*$. We denote the infinite prime of $k$ by $P_{\infty}$. The infinite primes of $K_D$ are characterised as follows.
\begin{lemma}[\protect{\cite[proposition 14.6]{Rosen}}]
		The infinite prime $P_{\infty}$ of $k$
		\begin{itemize}			
			\item splits in $K_D$ when $\deg(D)$ is even and $e=1$, 
			\item is inert in $K_D$ when $\deg(D)$ is even and $e=g$,
			\item ramifies in $K_D$ when $\deg(D)$ is odd.
		\end{itemize}
\end{lemma}
We call $K_D$ \emph{imaginary} when $P_{\infty}$ ramifies or is inert and \emph{real} when $P$ splits. The same distinction can also be made by the unit group, where imaginary fields have a finite unit group and real fields are of rank~$1$ over $\Z$.

Denote by $\O_D$ the integral closure of $\F_q[x]$ in $K_D$. The \emph{class group $C(D)$} (also denoted by $C$) is defined as the ideal class group $\O_D$. The $2$-part of $C(D)$ is completely determined by a sequence of integers, the \emph{$2^j$-ranks} for $j\geq 1$ as defined by \cref{def: 2^j-rank}. The sequence $r_2, r_4, r_8, \dots$ is weakly decreasing and zero for large enough $k$. As can be seen in \cref{def: 2^j-rank} there are two different ways of interpreting the $2^j$-rank.

Firstly, we can consider the quotient $C[2^j]/C[2^{j-1}]$, which describes the ideal classes of exact order $2^j$. Assuming we have a set of ideal classes generating $C[2^j]/C[2^{j-1}]$, the next rank can be calculated by checking which of these generators is the square of another ideal class. Indeed, if the class $A\in C[2^{j+1}]$ has exact order $2^{j+1}$ then $A^2$ has order $2^{j}$. In general it is hard to check which classes are squares, which is why we will also use the other expression.

For the second expression for the $2^j$-rank we use the dual group $\hat{C} = \operatorname{Hom}(C , \C)$. Let $H/K_D$ be the Hilbert class field of $K_D$, i.e . the maximal unramified abelian extension of $K_D$ where the infinite primes split completely \cite{Rosen2}. By the Artin map there is an isomorphism $C(D) \cong \Gal(H/K_D)$. For any $j\geq0$, the subgroup $C(D)^{2^j} \leq C(D)$ corresponds to an intermediate field $H_{2^j} \subset H$ such that $C(D)^{2^j} \cong \Gal(H/H_{2^j})$. It holds that
\[\Gal(H_{2^j}/K) \cong C/C^{2^j} \text{ and } \Gal(H_{2^j}/H_{2^{j-1}}) \cong C^{2^j}/C^{2^{j-1}} \cong \F_2^{r_{2^j}}.\]

A character $\chi \in \hat{C}[2^j]$ has a field of definition $L_{\chi}$ with a degree over $k$ that divides $2^j$. The kernel of $\chi$ contains $C^{2^j}$ and $L_{\chi}$ will be a subfield of $H_{2^j}$. In particular $C^{2^j}  = \cap_{\chi \in \hat{C}[2^j]} \ker(\chi)$ and $H_{2^j}$ is the compositum of the fields of definition of the characters $\chi \in \hat{C}[2^j]$. The extension $H_{2^{j-1}}\subset H_{2^{j}}$ is Galois with a Galois group of exponent $2$ and is therefore generated by $r_{2^j}$ characters in $\hat{C}[2^j]$ that have exact order $2^j$.
Similar to the first interpretation, the $2^{j+1}$ rank can be calculated by checking which characters in $\hat{C}[2^j]$ are squares of ones in $\hat{C}[2^{j+1}]$.

To find out which characters (or ideal classes) are squares of ones with higher torsion, we will need explicit expressions for the intermediate fields $H_{2^j}$. Writing $H_{2^{j}} = \prod_{i=1}^{r_{2^j}} H_{2^{j-1}}(\sqrt{A_i})$, we will explain in the next sections how to find these quadratic extensions $H_{2^{j-1}}(\sqrt{A_i})$ for $j=1,2,3$.

\section{Genus theory and R\'edei theory} \label{sec: genus and redei}
In this section we explain how to calculate the 2- and 4-rank of a hyperelliptic function field. While the results of this section are not new, we spend some time on the proof of \cref{thm: 2-rank} to explicitly construct bases of $C[2]$ and $\hat{C}[2]$ which are needed for the $8$-rank and results on governing fields in \cref{sec: 8-rank,sec: governing fields}.

\subsection{The 2-rank}\label{subsec: 2rank}
The Hilbert class field $H$ of $K_D$ is Galois over $k$ with Galois group $$\Gal(H/k) \cong \Gal(H/K_D)\rtimes \Gal(K_D/k) = C(D) \rtimes \langle \sigma \rangle$$ where $\sigma$ acts by inversion \cite{Peng}.  For the $2$-rank we focus on the subfield $H_2$, which is called the   \emph{genus field} of $K_D$. It is the maximal subextension of $H$ that is abelian over $k$ and it holds that
$$\Gal(H_2/\F_q(x)) = \Gal(H/\F_q(x))^{\text{ab}} = C/C^2 \rtimes \langle \sigma \rangle.$$ 
As the $2$-rank equals the dimension of $\Gal(H_2/H_1) = \Gal(H_2/K_D)$ as an $\F_2$-vector space, we obtain the $2$-rank by treating $\Gal(H_2/K_D) \subset \Gal(H_2/\F_q(x))$ as a subspace of codimension one. 

The quadratic characters generating $H_2$ are called \emph{genus characters}. They were first studied in the context of binary quadratic forms by Gauss (for number fields) \cite{Gauss} and Artin (for function fields) \cite{Artin}. A classification of the $2$-rank for imaginary function fields was obtained by Artin and extended to the real case by Zhang and S\'emirat. The author was unable to access \cite{Zhang}. A proof of its results,  based on the strategy explained in \cite[section~1]{Hu} is included in the author's master's thesis \cite[chapter 2]{Stokvis}.

\begin{theorem}[\cite{Artin}, \cite{Zhang}, \cite{Semirat}]\label{thm: 2-rank}
	Let $D = eQ_1\dots Q_{s_1}P_1\dots P_{s_2} \in \F_q[x]$ be a squarefree polynomial with the $Q_i$ irreducible of odd degree and the $P_j$ irreducible of even degree. Writing $s = s_1+s_2$, the $2$-rank of the class group $C(D)$ of the hyperelliptic function field $K_D$ is given by 
	\begin{itemize}
		\item $r_2= s-2$ if $e=1, \deg(D)$ is even and $s_1>0$,
		\item $r_2=s-1$ if  \hil$\deg(D)$ is odd\hr \:or \hil$e=1$ and $s_1=0$\hr \:or \hil$e=g, \deg(D)$ is even and $s_1>0$\hr,
		\item $r_2= s$ if $e=g$ and $s_1=0$.
	\end{itemize}
\end{theorem}
\begin{proof}[Proof (sketch)]
	We give two methods of proving the theorem. One classifies the ideal classes in $C[2]$ while the other finds characters in $\hat{C}[2]$.
	
	The first way of computing the $2$-rank is to find all ideal classes of $C[2]$, so-called \emph{ambiguous ideal classes}. It can be shown that (almost) all ambiguous ideal classes can be generated by the ramified prime ideals $\p_i \mid P_i$. Indeed $[\p_i^2] = [(P_i)]$ is a trivial class in $C(D)$. The generators are subject to the relation $\prod_{i=1}^{s}\p_i = (\sqrt{D})$, which would imply $r_2$ to be $s-1$. In the case that $K_D$ is real, the $2$-rank is brought down by one by the existence of a fundamental unit. It imposes an extra relation on the generators. When $s_1=0$, the rank is raised by one because of the existence of an extra generating ideal in $C(D)$. In this case, $D$ can be written as $D = c(U^2-gV^2)$ for some $c\in \F_q^*$ and $U, V\in \F_q[x]$ and the ideal $\mathfrak{A} =(V, U+\sqrt{D})$  (or swapping $U$ and $V$ when $c\notin \F_q^{*^2}$) of norm $U$ (or $V$) generates an \emph{irregular} ambiguous class in $C(D)$.
	
	Another way of computing $r_2$ is to give an explicit description of the genus field of $K_D$ as 
	\begin{equation}\label{eq: genus field}
		H_2 = \begin{cases}
			\F_q(x)(\sqrt{e}, \sqrt{P_1}, \dots, \sqrt{P_s}) &\text{ if $2\mid \deg(D)$ and $s_1=0$},\\
			\F_q(x)(\sqrt{e}, \sqrt{Q_1Q_2},\dots, \sqrt{Q_1Q_{s_1}}, \sqrt{P_1},\dots, \sqrt{P_{s_2}}) &\text{ if $2\mid \deg(D)$ and $s_1 \neq 0$},\\
			\F_q(x)(\sqrt{eQ_1}, \dots, \sqrt{eQ_{s_1}}, \sqrt{P_1},\dots, \sqrt{P_{s_2}}) & \text{ if  $2\nmid \deg(D)$}.
		\end{cases}
	\end{equation}
\end{proof}

The genus field mentioned in the proof is the compositum of fields of definition generated by quadratic characters $$\chi_{i}: \Gal(H_2/k) \to \Gal(k(\sqrt{A_i})/ k)$$ that give the Galois action on each $\sqrt{A_i}$ respectively. The $A_i$'s correspond to the generators of the genus field in \cref{eq: genus field}. When restricted to $C/C^2 =\Gal(H_2/K_D) \subset \Gal(H_2/k)$ we obtain a generating set of quadratic characters $\{\chi_1, \dots, \chi_{r_2}\} \subset \hat{C}[2]$. The Artin map thus induces a map to $\Gal(H_2/K_D)$ given by 
\begin{align}
	C(D) &\to \Gal(H_2/K_D)\\
	\notag [\p] &\mapsto \Art{\p}{ H_2}{K_D} = (\chi_1(\p), \dots, \chi_{r_2}(\p)) = \left(\Leg{A_i}{P}, \dots, \Leg{A_{r_2}}{P} \right),
\end{align}
where the latter are quadratic symbols, for example defined in \cite[chapter 3]{Rosen}. The two proofs give us a basis of ambiguous ideal classes spanning $C[2]$ and a basis of quadratic characters defining $\Gal(H_2/K_D)$. A quadratic character $\chi_i \in \hat{C}[2]$ has field of definition $k(\sqrt{A_i}, \sqrt{D/A_i})$. All quadratic characters (genus characters) can be created as $\chi_{D_1} = \sum_{i\in S} \chi_i$ for any subset $S \subset \{1,\dots, r_2\}$ having field of definition $k(\sqrt{D_1}, \sqrt{D_2}),$ where $D = D_1D_2$ and $D_1= \prod_{i\in S} A_i$. Note that $\chi_{D_1} = \chi_{D_2} \in \hat{C}[2]$.

\begin{remark}
	As can be seen in \cref{eq: genus field}, a splitting $D = D_1D_2$ corresponds to a genus character $\chi_{D_1}$ only when $D_2$ is monic and of even degree. This is required to ensure the splitting behaviour of $P_{\infty}$ is consistent in $K_D$ and $k(\sqrt{D_1}, \sqrt{D_2})$. When $D$ is monic, i.e. $e=1$, the genus field may contain a trivial extension through the splitting $D = 1\cdot D$. The corresponding character $\chi_e$ is trivial and can thus be ignored.
\end{remark} 

\subsection{The 4-rank}
The following results have been introduced in \cite{Wittmann}. Consider the map 
\begin{equation}\label{eq: phi4}
	\phi_4: C[2] \to C/C^2
\end{equation}
incduced from the quotient map. An ideal class in $C[2]$ is also the square of an element in $C[4]$ if and only if it lies in $\ker(\phi_4)$. From the two proofs of \cref{thm: 2-rank} we have explicit generators for $C[2]$ and $\hat{C}[2]$. Writing $D = eQ_1\dots Q_{s_1}P_1\dots P_{s_2}$ with $s=s_1+s_2$, there exists an irregular ambiguous ideal $\mathfrak{A}$ in $C[2]$ when $s_1=0$. Thus the number of generators of $C[2]$ is given by
\begin{equation}\label{eq: delta def}
	\delta := \begin{cases}
		s+1 &\text{ if  $\deg(D)$ is even and $s_1=0$},\\
		s &\text{ otherwise,} 
	\end{cases}
\end{equation}
and we can define a surjection
\begin{equation}\label{eq: alpha for R_4}
	\alpha: \F_2^{\delta} \twoheadrightarrow C[2], \quad e_i \mapsto \begin{cases}
	[\p_i] &\text{ if } 1\leq i\leq s,\\
	[\mathfrak{A}] &\text{ if $i=s+1$ (only when $\delta = s+1$).}	
\end{cases}
\end{equation}
The first $s$ generators are subject to one or two relations, depending on whether $K_D$ is real or imaginary. The genus characters in $\hat{C}[2]$ have a set of $r_2+1$ generators subject to one relation given by $\chi_{1} = \prod_{i=1}^{r_2+1} \chi_{i}$. Using the inclusion map 
\begin{equation}
	\beta: C/C^2 \cong \Gal(H_2/K_D) \hookrightarrow \Gal(H_2/\F_q(x)) \cong \{\pm1\}^{r_2+1}
\end{equation} we construct the so-called \emph{R\'edei map} 
\begin{equation}\label{eq: Redei map}
	R_4: \F_2^{\delta} \xrightarrow{\alpha} C[2] \xrightarrow{\phi_4} C/C^2 \xrightarrow{\beta}
	\{\pm1\}^{r_2+1}.
\end{equation}
It follows that the $4$-rank can be calculated as follows.
\begin{theorem}[\protect{\cite[theorem 3.1]{Wittmann}}]\label{thm: 4-rank}
	Let $D \in \F_q(x)$ be a squarefree polynomial and $K_D$ the corresponding hyperelliptic function field. The $4$-rank of the class group $C(D)$ of $K_D$ is given by
	\[r_4 = \dim \ker(\phi_4) = r_2 - \rank R_4.\]
\end{theorem}
Calculating the $4$-rank means determining the rank of a $\delta \times (r_2+1)$-matrix $R_4$. The entries are quadratic symbols, which written multiplicatively are given by
\begin{equation}\label{eq: R4 entries}
	(R_4)_{ij} = \chi_i(\p_j) = \Leg{A_i}{P_j} \text{ (whenever $P_i$ and $A_j$ are relatively prime),}
\end{equation}
where the $\p_j$ are ramified primes generating $C[2]$ (and potentially one for the irregular ambiguous ideal class which is not given by \cref{eq: R4 entries}) and $\chi_i$ gives the Galois action on $\sqrt{A_i}$ in $H_2$.

\section{The 8-rank}\label{sec: 8-rank}
For the $8$-rank of $C(D)$, we will use the same setup as for the $4$-rank. Many results in this section also hold for number fields and can be found in \cite{Stevenhagen1}. Consider the map
\begin{equation}
	\phi_8: \ker(\phi_4) = C[2]\cap C^2 \to C^2/C^4
\end{equation}
induced from the quotient map. As ideal classes of exact order $8$ correspond to fourth powers in $C[2]$, the $8$-rank equals the dimension of the kernel of $\phi_8$. We can make a restriction to the R\'edei map in \cref{eq: Redei map} to obtain
\begin{equation}\label{eq: 8-rank map}
	R_8: \ker(R_4) \to C^2\cap C[2] \xrightarrow{\phi_8} C^2/C^4 \cong \Gal(H_4/H_2) \cong \prod_{i=1}^{r_4} \Gal(H_2(\sqrt{\beta_i})/H_2) \cong \{\pm1\}^{r_4}.
\end{equation}
The first part of this map is induced from \cref{eq: alpha for R_4} and is surjective. The $8$-rank can now be calculated as follows. 
\begin{theorem}
	Let $D \in \F_q(x)$ be a squarefree polynomial and $K_D$ the corresponding hyperelliptic function field. The $8$-rank of the class group $C(D)$ of $K_D$ is given by
	\[r_8 = \dim \ker(\phi_8) = r_4 - \rank R_8.\]
\end{theorem}
There is no obvious choice of basis for $\ker(R_4)$ and $\Gal(H_4/H_2)$ as $\F_2$-vector spaces anymore, which makes the computation much harder. Let us try to make $\Gal(H_4/H_2)$ more explicit. The Hilbert subfield $H_4/K_D$ is the compositum of the fields of definition for all characters $\psi \in \hat{C}[4]$. The extension $H_4/H_2$ is generated by the characters in $\hat{C}[4]/\hat{C}[2]$. Such a quartic character $\psi$ becomes a genus character when squared. Classification of squares in the group of genus characters goes as follows. A version and proof of this lemma exist also for number fields \cite[lemma~4.2]{Stevenhagen1}.

\begin{lemma}\label{lem: decomposition 2nd type ffields}
	Let $K_D$ a hyperelliptic function field with ideal class group $C(D)$.	For a non-trivial  quadratic character $\chi \in \hat{C}[2]$ with field of definition $E = k(\sqrt{D_1}, \sqrt{D_2})$, having $\chi = \psi^2 \in \hat{C}^2$ for a quartic character $\psi$ is equivalent to any of the following:
	\begin{enumerate}
		\item \label{item 1}there exists a cyclic quartic extension $K_D \subset F$ inside $H$ containing $E$,
		\item \label{item 2}all ambiguous ideals of $K_D$ split completely in $K_D\subset E$,
		\item \label{item 3}for $i=1,2$ and all irreducible polynomials $P\mid D_i$ we have $\Leg{D/D_i}{P}=1$. When it exists, the extra irregular ambiguous ideal $\mathfrak{A}$ of norm $N(\mathfrak{A})$ satisfies $\Leg{D_1}{N(\alpha)} = \Leg{D_2}{N(\alpha)}=1$.
	\end{enumerate}  
\end{lemma}
\begin{proof}
	The character $\chi$ belonging to the splitting $D = D_1D_2$ has $E$ as its field of definition. Requiring $\psi = \chi^2$ for a quartic character $\psi$ is indeed equivalent to the existence of a cyclic unramified extension $K\subset F$ of degree $4$ that contains $E$ as in $(i)$. The extension $F$ is contained in $H_4$.
	
	The fact that $\chi = \psi^2$ is also equivalent to $\chi$ vanishing on $C[2]$. The group $C[2]$ is generated by the ramifying primes $\p_i$ of $K_D$ and the potential irregular ideal class. The fact that $\chi$ is trivial on $C[2]$ is equivalent to saying that under the Artin map all these primes are trivial. In other words, all the ramified primes (and irregular ideal) of $K_D$ split in the extension $K\subset E$ as in \cref{item 2}. Equivalently, it holds that $\Leg{D_i}{p}=1$ whenever $p\mid D$ and $p\nmid D_i$ as in \cref{item 3}, because the ramified primes are the ones dividing $D$. In \cref{item 3} one has to include the condition on the irregular ideal as well.
\end{proof}
\begin{remark}
	Whenever there is no irregular ambiguous ideal, note that \cref{item 3} of \cref{lem: decomposition 2nd type ffields} is equivalent to $D_1$ and $D_2$ having a trivial Hilbert symbol for any prime $P$. Also note that the decompositions $D=D_1D_2 = D_2D_1$ are considered equal, as they generate the same quartic extension.
\end{remark}

\begin{definition}
		A decomposition $D = D_1D_2$ with the properties of \cref{lem: decomposition 2nd type ffields} is called a \emph{decomposition of the second type}.
\end{definition}

To represent $R_8$ as a matrix, we first choose a basis of $\ker(R_4)$ and write $[\m_j] \in C^2\cap C[2]$ for the corresponding ideals under the map $\alpha$ from \cref{eq: alpha for R_4}. We obtain a basis of size $\dim \ker(R_4) = \dim \ker(\phi_4) + (\delta - r_2) = r_4 + (\delta-r_2)$. These classes span $C^2\cap C[2]$ and are subject to $(\delta - r_2)$ relations.
We also choose a set $\{\psi_i\}_{i=1}^{r_4}$ of quartic characters in $\hat{C}[4]$ that form a basis of $\hat{C}[4]/\hat{C}[2]$. Their fields of definition $F_{i}$ are cyclic extensions of $K_D$ of degree $4$ and the quadratic extensions $H_2F_i/H_2$ span $H_4$. The intersection $F_i \cap H_2 = K_D(\sqrt{D_i})$ is a quadratic extension of $K$ such that $D = D_i\cdot \frac{D}{D_i}$ is a decomposition of the second type with genus character $\chi _i= \psi_i^2$. The map $R_8$ sends a class $[\m_j]$ to $R_8([\m_j]) \in \F_2^{r_4}$, of which the $i$-th component is $\Art{\m_j}{H_2F_i}{H_2}$. We can therefore represent $R_8$ by a matrix $(\eta_{ij})_{i,j} \in \text{Mat}_{r_4\times r_4+(\delta -r_2)}(\{\pm1\})$ with
\begin{equation}\label{eq: explicit R_8}
	\eta_{ij} = \Art{\m_j}{H_2 F_{i}}{K_D} \in \Gal(H_2F_{i}/H_2)\cong \{\pm1\} .
\end{equation}
Note that the $[\m_j]$ are classes in $C[2] \cap C^2$ so that $\Art{\m_j}{H_2F_i}{K_D}$ is the identity on $H_2$. Hence the symbol in \cref{eq: explicit R_8} can indeed be taken in $\Gal(H_2F_{i}/H_2)$. 

Each of the quartic characters $\psi$ in our chosen basis has a field of definition $F$ that is a quartic extension of $K_D$ and a quadratic extension of $E = K_D(\sqrt{D_1}) = k(\sqrt{D_1}, \sqrt{D_2})$. The splitting $D=D_1D_2$ is a decomposition of the second type and $E$ is the field of definition of the genus character $\chi_{D_1} = \psi^2$. Writing $F = E(\sqrt{\beta})$ with some $\beta \in E$, we can obtain an expression for $\beta$ by the following lemma.

\begin{lemma}[\protect{\cite[corollary 5.2]{Stevenhagen1}}] \label{lem: finding F}
	Let $Q$ be a field of characteristic different from $2$ and $a,b\in Q^*$ such that $a,b,ab \not\equiv 1 \mod{Q^{*^2}}$. Then a quadratic extension $E = Q(\sqrt{a} , \sqrt{b})\subset F$ is cyclic over $Q(\sqrt{ab})$ and dihedral of degree $8$ over $Q$ if and only if there exists a non-zero solution $(X,Y,Z) \in Q^3$ to the equation 
	\[X^2 - aY^2 - bZ^2 = 0\]
	such that for $\beta = X+Y\sqrt{a} \in Q(\sqrt{a})$ and $\alpha = 2(X+Z\sqrt{b}) \in Q(\sqrt{b})$ we have 
	\[F = E(\sqrt{\beta}) = E(\sqrt{\alpha}).\]
	Any other quadratic extension of $E$ that is dihedral over $Q$ of degree $8$ is of the form $F_t = E(\sqrt{t\beta})$ for some $t \in Q^*/\langle a, b, Q^{*^2}\rangle$, which we call a \emph{twist} of $\beta$.
\end{lemma}
As can be seen in \cref{eq: explicit R_8} the entry $\eta_{ij}$ of $R_8$
depends on the choice of ideal $\m_j$ and extension $F_i$. Because $[\m_j] \in C[2]$ we can also take its conjugate (i.e. inverse) as representative, making $\eta_{ij}$ only depend on $M_i= N_{K_D/\F_q(x)}(\m_j) \in \F_q[x]$. By \cref{lem: finding F} the quartic extension $F_i$ only depends on the decomposition $D = D_1D_2$ of the second type. Writing $\eta_{ij} =\Art{\m_j}{H_2F_i}{H_2}=[D_1, D_2,M_j]$, we obtain a first definition of the R\'edei symbol for function fields.
\begin{definition}[R\'edei symbol, version 1]\label{def: early Redei symbol}
	Let $D = D_1D_2$ be a decomposition of the second type, $F/K_D$ a quartic extension corresponding to \cref{item 1} in \cref{lem: decomposition 2nd type ffields}, and $M\mid D$ a polynomial that is the norm of an ideal $\m \in K_D$ such that $[\m]\in C^2\cap C[2]$. Then the \emph{R\'edei symbol} of $D_1, D_2$ and $M$ is given by
	\begin{equation}
		[D_1, D_2, M] := \Art{\m}{H_2F}{H_2} \in \{\pm1\}.
	\end{equation}
\end{definition}

\begin{example}
	Let us illustrate how to calculate the $2$-, $4$- and $8$-rank in the case that $D=QP$ is the product of two primes where $Q$ has odd degree and $P$ has even degree. The function field $K = \F_q(x, \sqrt{QP})$ is imaginary and by \cref{thm: 2-rank} the $2$-rank equals $1$. The quadratic characters generating $\hat{C}[2]$ are $\chi_Q$ and $\chi_P$. They are in fact equal as their product $\chi_Q\chi_P$ is the trivial character. The prime ideals generating $C[2]$ are $\q$ and $\p$ lying above $Q$ and $P$ respectively, also with the relation that $[\q\p]$ is the trivial class. By \cref{thm: 4-rank} the $4$-rank equals $r_2$ minus the rank of $R_4$. The R\'edei matrix in this instance is a $(2\times 2)$-matrix of rank $0$ or $1$, where all the entries are equal to the same quadratic symbol $\chi_P(\q) = \Leg{Q}{P}$ as a result of quadratic reciprocity \cite[theorem 3.5]{Rosen}. Thus the $4$-rank is $1$ if and only if $\Leg{Q}{P}=1$. Note that the genus field is $H_2 = k(\sqrt{Q}, \sqrt{P}).$

	Assuming that $\Leg{Q}{P}$=1, we want to calculate the $8$-rank. The only decomposition of the second type is the splitting $D_1 = Q, D_2=P$, corresponding to the character $\chi_Q$ (which equals $\chi_P$). A generator of $C^2\cap C[2]$ is the ideal class $[\q] = [\p]$ and so the map $R_8$ from \cref{eq: 8-rank map} is determined by the R\'edei symbol $[Q,P, P] = \Art{\p}{H_2(\sqrt{\beta})}{K_D}$. The extension $H_2(\beta)/H_2$ can be found by \cref{lem: finding F}: we know that the conic $X^2 = QY^2 + PZ^2$ has a rational solution by \cref{lem: decomposition 2nd type ffields} and we can set $\beta = X+\sqrt{Q}Y$.
	
	The prime ideal $\p$ in $K_D$ splits in $H_2/K_D$ as $\p = \p_1\p_2$, where we set $\p_1=(\sqrt{P}, X-Y\sqrt{Q})$. As the extension $H_2(\beta)/K_D$ is abelian, we only have to look at the splitting behaviour of $\p_1$ in $H_2(\beta)/H_2$ to compute the Artin symbol $\Art{\p}{H_2(\sqrt{\beta})}{K_D}$. It follows that
	$$r_8= 1 \text{ if and only if } \beta \in H_2 \text{ is a square modulo } \p_1.$$
	Note that there is an isomorphism $\F_q[x, \sqrt{Q}, \sqrt{P}]/\p_1 \cong \F_q[x]/(P)$ by the identification $\overline{\sqrt{Q}} \mapsto \bar{X}/\bar{Y}$. Thus $\beta$ is a square modulo $\p_1$ if and only if $2X$ is a square modulo $P$. Since $P$ has even degree all elements in $\F_q^*$ are squares modulo $P$. Using the reciprocity law we also know that $\Leg{Y}{P} = \Leg{P}{Y} = \Leg{X^2}{Y}=1$, meaning $\bar{Y}$ is a square. Hence $\bar{\beta}$ is a square if and only if $\overline{\sqrt{Q}} = \bar{X}/\bar{Y}$ is a square in $\F_q[x]/(P).$ We conclude that 
	\[r_8 = 1 \text{ if and only if $P$ splits in } k(\sqrt{g}, \sqrt[4]{G}).\]

\end{example}

\section{Minimally ramified extensions and R\'edei reciprocity}\label{sec: redei reciprocity}

The matrix entries for the $R_8$-map in \cref{eq: 8-rank map} are defined by constructing extensions of $K_D$ that are completely unramified. This condition can actually be relaxed into a property we will call minimally ramified and allows for a more general definition of the R\'edei symbol. This general symbol obeys a reciprocity law that can be used to show the existence of governing fields. The theory is based on the version for number fields \cite{Stevenhagen1}, which relies on methods by R\'edei \cite{Redei1} and Corsman \cite{Corsman}.

The R\'edei symbol in \cref{def: early Redei symbol} takes arguments $D_1, D_2, M$, where $D_1$ and $D_2$ are squarefree. When we assume there is no irregular ambiguous ideal class (which will be one of the assumptions to be able to construct governing fields in \cref{sec: governing fields}), $M$ will also be squarefree as it is the norm of a ramified prime in $C[2]\cap C^2$. The general R\'edei symbol $[A,B,C]$ will therefore take its arguments in $k^*/k^{*2}$. Each class $\bar{A}\in k^*/k^{*2}$ contains a unique (up to multiplication by $\F_q^*$) squarefree element $A \in \F_q[x]$ which we will use to represent the class.  For non-trivial squarefree polynomials $A,B \in k^*$, the extension
\[k(\sqrt{AB}) = K \subset E = k(\sqrt{A}, \sqrt{B})\]
is quadratic and unramified at all finite primes not dividing $\gcd(A,B)$, and $E/k$ is ramified at the finite primes dividing $AB$ and at infinity if $A$ and/or $B$ have odd degree. To be able to construct minimally ramified extensions, we also require $A$ and $B$ to have trivial quadratic Hilbert symbols $(A,B)_P = 1$ for all primes $P$ (including the infinite prime) so that
\begin{equation}\label{eq: norm solvability}
	X^2-AY^2 - BZ^2=0
\end{equation}
has a non-trivial solution in $k$ by the Hasse-Minkowski principle  \cite[theorem 3.1]{Lam}. For a definition of Hilbert symbols for global fields and a reciprocity law we refer to \cite[chapter XII]{Artin-Tate}.

\begin{remark} Recall that the infinite prime ramifies in the quadratic extension $k(\sqrt{A}) /k$ precisely when $A$ has odd degree. This divisibility by infinity will be included in the greatest common divisor, denoted by $\gcd_{\infty}(\cdot\, ,\cdot)$. For $A,B \in \F_q[x]$ we say that $P_{\infty}$ divides $\gcd_{\infty}(A,B)$ if and only if $A$ and $B$ have odd degree.
\end{remark}

\subsection{Minimally ramified extensions}\label{subsec: minimally ramified extensions}
The non-trivial solvability of \cref{eq: norm solvability} implies the existence of a cyclic quartic extension
\[k(\sqrt{AB}) = K \subset F = E(\sqrt{\beta}) = E(\sqrt{\alpha})\]
by \cref{lem: finding F}. In the special case that $A=B$ there is a similar result by the following lemma.

\begin{lemma}[\protect{\cite[corollary 5.3]{Stevenhagen1}}] \label{lem: finding F special case}
	Let $Q$ a field of characteristic different from $2$. For $E=Q(\sqrt{a})$ with $a \not\equiv 1 \mod{Q^{*^2}}$, a quadratic extension $E\subset F$ is cyclic over $Q$ if and only if there exists a non-zero solution $(X,Y,Z) \in Q^3$ to 
	\[X^2-aY^2-aZ^2=0\]
	such that we have $F = E(\sqrt{\beta})$ for $\beta =X+Y\sqrt{a} \in Q(\sqrt{a})$ of norm $\beta\beta' \in aQ^2$. Any other such extension is of the form $F_t = E(\sqrt{t\beta})$ for some unique $t \in Q^*/\langle a, Q^{*2} \rangle$. 
\end{lemma}

Suppose there exists a cyclic extension $F/K$ coming from a solution to \cref{eq: norm solvability}. The primes $\p$ in $K$ dividing both $A$ and $B$ are ramified in $E/K$ and these $\p$ will be totally ramified in $F/K$, because $\beta$ and $\alpha$ have norms (up to squares in $k^*$) $A$ and $B$ respectively. This might include the infinite primes in $K$ when $A$ and $B$ are both of odd degree. We cannot avoid ramification of these primes, but we will be able to find an $F$ that allows almost no other ramification.

\begin{definition}\label{def: minimal ramification}
	For non-trivial $A,B \in k^*/k^{*2}$ satisfying $(A,B)_P=1$ for all primes $P$, we call the cyclic extension $F/K$ resulting from a non-trivial solution to \cref{eq: norm solvability} \emph{minimally ramified over $E=k(\sqrt{A}, \sqrt{B})$} if it is either
	\begin{enumerate}
		\item unramified over all primes $P \nmid \gcd_{\infty}(A,B)$ when at least one of $\deg(A), \deg(B)$ is odd or $\gcd_{\infty}(A,B)$ has a finite prime of odd degree, or
		\item ramified over at most one finite prime $P \nmid \gcd_{\infty}(A,B)$ when $\deg(A)$ and $\deg(B)$ are both even and $\gcd_{\infty}(A,B)$ contains only finite primes of even degree.\label{item: extra ramified}
	\end{enumerate} 
\end{definition}

As we want to prove a reciprocity law on three entries, they have to pairwise satisfy the above condition. We obtain the following definition.

\begin{definition}\label{def minimal ramification over C}
	Let $A,B,C \in k^*/k^{*2}$ be non-trivial classes represented by squarefree polynomials  $A,B,C \in \F_q[x]$ that satisfy
	\begin{align}
		&\operatorname{gcd}_{\infty}(A,B,C)=1, \label{redei cond 1}\\
		&(A, B)_P = (A,C)_P = (B, C)_P = 1 \quad\text{ for all primes $P$.} \label{redei cond 2}
	\end{align}
	We call a minimally ramified extension $F$ as in \cref{def: minimal ramification} \emph{minimally ramified over $k(\sqrt{A}, \sqrt{B})$ with respect to $C$} if one of the following holds:
	\begin{enumerate}
		\item \label{good case} $k(\sqrt{AB}) \subset F$ is unramified outside $\gcd_{\infty}(A,B)$, or 
		\item \label{bad case}$\deg(A)$ and $\deg(B)$ are both even, $\gcd_{\infty}(A,B) $ contains only irreducible parts of even degree, and there is an extra ramified prime $Q \in \F_q[x]$ as in \cref{item: extra ramified} from \cref{def: minimal ramification} satisfies
		\begin{equation}\label{eq: extra ramification requirement}
			Q \text{ has odd degree, and }\:\Leg{Q}{P}=1 \text{ for all finite primes $P$ dividing $ABC$.}
		\end{equation}
	\end{enumerate}
\end{definition}

\begin{lemma}\label{lem: min ramified exists}
	For non-trivial $A,B,C  \in k^*/k^{*2}$ satisfying \cref{redei cond 1} and \cref{redei cond 2} there exists an extension $F_{A,B}$ that is minimally ramified over $E$ with respect to $C$. 
\end{lemma}
\begin{proof}
	Let $(X,Y,Z)$ be a primitive solution in $\F_q[x]$ to \cref{eq: norm solvability} and $F/K$ the resulting quartic extension. First consider a finite prime $P$ that does not divide $A B$. Then $P$ cannot divide both $\beta = X+Y\sqrt{A}$ and $\alpha = 2(X+Z\sqrt{B})$, as it would imply that $P$ divides both $N(\beta) = BZ^2$ and $N(\alpha)= AY^2$ making it a common divisor of $X,Y$ and $Z$ contradicting our assumption. Thus without loss of generality $\beta$ is a unit at a prime above $P$ in $k(\beta)$, implying that the extension $K \subset F = k(\sqrt{AB}, \sqrt{A}, \sqrt{\beta})$ is unramified at a prime $\p$ dividing $P$.
	
	Next, consider a finite prime $P$ dividing $AB$. Without loss of generality assume it divides $A$. As $P\nmid B$ the field $E = K(\sqrt{B}) = K_A(\sqrt{B})$ is a quadratic extension unramified at $P$ of both $K$ and $K_A$. Thus $F/K$ is unramified in primes above $P$ if and only if $F/K_A$ is. Write $F = K_A(\sqrt{\beta}, \sqrt{\beta'})$ with $\beta\beta' = X^2-AY^2 = BZ^2$. A prime $\p\mid  P$ is unramified in $F/K_A$ if and only if $\beta$ and $\beta'$ have an even valuation at $\p$. Since we assumed that $P\mid A$, there is a unique $\p\mid P$ in $K_A$ with ramification index $2$. We obtain
	\[2\ord_{\p}(\beta) = \ord_{\p}(\beta\beta') = 2\ord_P(BZ^2) = 2\cdot 2 \ord_P(Z) \equiv 0 \mod4,\]
	which shows that $P$ is unramified in $F/K$. 
	
	For the infinite prime $P_{\infty}$ of $k$ we will do a similar construction. Let $\p_{\infty}$ be a prime in $K_A$ dividing $P_{\infty}$ with ramification index $e$. We obtain
	\[2\ord_{\p_{\infty}}(\beta) = \ord_{\p}(\beta\beta') = e(\p_{\infty}/P_{\infty})\ord_{P_{\infty}}(BZ^2) = e (\ord_{P_{\infty}}(B) + 2\ord_{P_{\infty}}(Z)).\]
	Assume first that $P_{\infty}$ divides $A$ but not $B$ (without loss of generality). Then $A$ is of odd degree and $P_{\infty}$ ramifies in $K_A/k$. As $B$ is of even degree $\ord_{P_{\infty}}(B)\equiv 0 \mod2$ and we find $\ord_{\p_{\infty}}(\beta) \equiv 0 \mod2$. When both $\deg(A)$ and $\deg(B)$ are odd we allow ramification of the inifinite prime. In the case they are both even the previous valuation argument does not hold, as the ramification index of $\p_{\infty}\mid P_{\infty}$ in $K_A$ is $1$. The result is that $\ord_{\p_{\infty}}(\beta)$ is even if and only if $\deg(B)/2 \equiv \deg(Z) \mod{2}$. As this is not guaranteed by the primitive solution to \cref{eq: norm solvability}, we need to apply a quadratic twist to $F_{A,B}$ whenever $\ord_{\p_{\infty}}(\beta)$ is odd. Twisting with an odd degree irreducible polynomial $Q \in \F_q[x]$ causes the desired change in parity and hence non-ramification of $\p_{\infty}$. However, it also causes potential ramification at primes $\q\mid Q$ in $k(\sqrt{AB})$. If there is an odd degree irreducible in $\gcd_{\infty}(A,B)$ we can take that one. Otherwise we have to allow for some $Q$ outside $\gcd_{\infty}(A,B)$ to be ramified. By the Chebotarev density theorem we can take it to satisfy \cref{eq: extra ramification requirement}.
\end{proof}

From the proof above it is not at all clear when the extra twist at $Q$ is required in \cref{bad case} of \cref{def minimal ramification over C}. We have for example the following lemma, which shows that there may exist extension $F_{A,B}$ that are unramified outside $\gcd_{\infty}(A,B)$ even in the `bad case'.

\begin{lemma}\label{lem: bad case still good}
	Let $A,B,C \in k^*/k^{*^2}$ satisfy the conditions of \cref{def minimal ramification over C} such that $A$ and $B$  are both non-trivial, have even degree,  and $G:=\gcd_{\infty}(A,B)$ contains only irreducible parts of even degree.  Then there exists an extension $F_{A,B}$ that is minimally ramified with respect to $C$ and unramified outside $\gcd_{\infty}(A,B)$ if one of the following holds:
	\begin{itemize}[noitemsep]
		\item  $A$ and $B$ are coprime and $D=AB$ is a decomposition of the second type,
		\item  $B\mid A$ and $D = A/B \cdot B$ is a decomposition of the second type,
		\item	$A\mid B$ and $D = B/A \cdot A$ is a decomposition of the second type.
	\end{itemize}
\end{lemma}
\begin{proof}
	In each of these cases, we can use \cref{item 1} of \cref{lem: decomposition 2nd type ffields} to find an extension $F_{A,B}$ that lies in the Hilbert class field of $D$ and is therefore unramified over $k(\sqrt{D})$. Because $k(\sqrt{AB}) \subset k(\sqrt{D})$ is unramified outside $\gcd(A,B)$ we can take $F_{A,B}$ as the desired extension.
\end{proof}
\begin{remark}
	In the case that there is at least one odd degree irreducible polynomial dividing $AB$, the conditions of \cref{lem: bad case still good} can be changed to 
		\begin{itemize}[noitemsep]
		\item  $A$ and $B$ are coprime, or
		\item  $B\mid A, \Leg{A/B}{P}=1$ for all $P\mid B$ and the class group of $k(\sqrt{A})$ has $r_2>0$, or 
		\item	$A\mid B, \Leg{B/A}{P}=1$ for all $P\mid A$ and the class group of $k(\sqrt{B})$ has $r_2>0$.
	\end{itemize}
	Note that in each case we don't have to concern ourselves with an irregular ambiguous ideal. Because we assumed that $(A,B)_P=1$ for all $P\leq \infty$ the requirements above ensure that \cref{item 3} of \cref{lem: decomposition 2nd type ffields} is satisfied.
\end{remark}

\begin{example}\label{ex: hard case}
	Contrary to number fields \cite[Lemma 7.7]{Stevenhagen1}, there does not always exists a cyclic quartic extension of $k(\sqrt{AB})$ that is unramified outside $\gcd(A,B)$. Consider for example monic polynomials $A = x^4+2x^3+x+1$ and $B= +x+2$ over the finite field $\F_3$.
	There exists a solution to 
	\[X^2=AY^2+BZ^2\]
	by taking $X=x^3+x^2+x, Y=1, Z=x^2+2x+1$. The extension $k(\sqrt{A}, \sqrt{\beta}, \sqrt{\beta'})/k(\sqrt{AB})$ is ramified at the infinite prime. Because $\gcd(A,B)=1$ any twist of $\beta$ (see \cref{lem: finding F}) will cause new ramification. The extra ramified prime that has to satisfy \cref{eq: extra ramification requirement} is thus a necessity for function fields. More examples of this type can be generated using the Magma code in \cref{app: magma}
\end{example}

\subsection{The R\'edei symbol}
Using the existence of minimally ramified extensions we can define the general version of the R\'edei symbol.

\begin{definition}[R\'edei symbol, version 2]\label{def: Redei symbol}
	Let $A,B,C \in k^*/k^{*2}$ be non-trivial classes represented by squarefree polynomials  $A,B,C \in \F_q[x]$ that satisfy
	\begin{align}
		&\text{gcd}_{\infty}(A,B,C)=1,\\% \label{redei cond 1}\\
		&(A, B)_P = (A,C)_P = (B, C)_P = 1 \quad\text{ for all primes $P$.}% \label{redei cond 2}
	\end{align}
	Let $F_{A,B}/K$ be minimally ramified over $E$ with respect to $C$. Then the \emph{R\'edei symbol }
	\[[A,B,C] \in \Gal(F_{A,B}/E) \cong \{\pm1\}\]
	is defined as
	\begin{equation}\label{eq: Redei symbol}
		[A,B,C] = \begin{cases}
			\Art{\cc}{F_{A,B}}{K} &\text{ if } 2\mid \deg(C),\\
			\Art{\infty\cc}{F_{A,B}}{K} &\text{ if } 2\nmid \deg(C),
		\end{cases}
	\end{equation}
	where the ideal $\cc$ is an integral $\O_K$-ideal of norm $C$ and $\infty$ is an infinite prime of $K$. If one of $A,B,C$ is trivial in $k^*/k^{*2}$ we define $[A,B,C]=1$.
\end{definition}
Note that this definition also captures the special case $A=B$, where the minimally ramified extension $F$ is cyclic over $k$. We now check that \cref{def: Redei symbol} is well-defined, i.e. that it does not depend on the choice of $\cc$ and that $[A,B,C]$ indeed lies in $\Gal(F_{A,B}/E)$.

Let $A,B,C$ be non-trivial squarefree polynomials satisfying the conditions of \cref{def: Redei symbol}. Let $F/K$ be minimally ramified over $E = k(\sqrt{A}, \sqrt{B})$ with respect to $C$, which exists by \cref{lem: min ramified exists}. Consider a prime $P$ dividing $C$. Then $P$ splits or ramifies in $K_A$ and $K_B$ by \cref{redei cond 2}, and is unramified in at least one of them by \cref{redei cond 1}. Thus a prime $\p_K \mid  P$ in $K$ will have degree $1$, and is split in the extension $E/K$. The prime $\p_K$ is also unramified in $F/K$ since $F$ is minimally ramified with respect to $C$. Thus $\Art{\p_K}{F}{K} \in \Gal(F/E)$ is well-defined. Since $\Gal(F/E)$ is contained in the center of $\Gal(F/k)$, $\Art{\p_K}{F}{K}$ is also well-defined as element in $\Gal(F/k)$ and only depends on $F$ and $P$, and not on the choice of $\p_K \mid  P$. For such a $P$ dividing $C$ we can therefore define the \emph{local R\'edei symbol}
\begin{equation}\label{eq: local artin symbol}
	[A,B,C]_{F,P} := \Art{\p_K}{F}{K} \in \Gal(F/E).
\end{equation}
Whenever $P$ does not divide $C$ we set $[A,B,C]_{F,P} := \id_F$. The local symbol is also defined for the infinite prime. The R\'edei symbol can now be written as a product of its local parts
\[ [A,B,C] = \prod_{P\leq \infty} [A,B,C]_{F,P} \in \Gal(F/E).\]
There are only finitely many non-trivial elements in the infinite product, namely the parts for $P$ dividing $C$, so the infinite product is a well-defined element in $\Gal(F/E)$. Moreover, this shows that the R\'edei symbol does not depend on the choice of ideal $\cc$ but only on $C$.

The local R\'edei symbols can be interpreted in the following way. A prime $\p_K$ splits in $E/K$ whenever $P$ divides $C$, so $[A,B,C]_{F,P}$ can be calculated by the Artin symbol $\Art{\p_E}{F}{E}$ for some prime $\p_E \mid P$ in $E$, implying that $[A,B,C]$ lies in $\Gal(F/E)$. Since $\p_E$ is unramified in $F/E$, its norm $N_{E/K_A}(\p_E) = \p$ in $K_A$ is unramified in at least one of the extensions $K_A(\sqrt{\beta})$ and $K_A(\sqrt{\beta'})$. Replacing $\p$ with a conjugate prime in $K_A$ if necessary, we may assume that $\p$ is unramified in $K_A(\sqrt{\beta})/K_A$. The $P$-part of the R\'edei symbol can now be calculated as
\[[A,B,C]_{F,P} = \Art{\p}{K_A(\sqrt{\beta})}{K_A}.\]
Essentially, $[A,B,C]_{F,P}$ is the quadratic symbol $\Leg{\beta}{\p}$ in the quadratic field $K_A$. We can also write the $P$-part of the R\'edei symbol as the Hilbert symbol

\begin{equation}\label{eq: redei as hilbert symbol}
	[A,B,C]_{F,P} = (\beta, \pi)_{\p},
\end{equation}
where $\pi$ is a uniformiser of $K_A$ at $\p$. 

\begin{remark}
	When $\deg(A)$ and $\deg(B)$ are even, we may end up in the special case of adding a quadratic twist $\beta \to Q\beta$ for some irreducible $Q\in \F_q[x]$. By the assumption that $\Leg{Q}{P}=1$ for all finite $P$ dividing $C$, we know obtain $(Q, \pi)_{\p} = 1$ for any choice of $\p\mid P$ in $K_A$. Thus the twisting by $Q$ has no influence on the local R\'edei symbols for any finite prime $P$.
\end{remark}

\begin{remark}
	The infinite part of the R\'edei symbol only needs to be calculated when $C$ has odd degree and thus at least one of $A,B$ has even degree. A prime at infinity in $K$ is also unramified in $F/K$. Thus $K_A(\beta)/K_A$ is also unramified at infinity (after potentially taking conjugates) and by the same reasoning as above the $\infty$-part of the R\'edei symbol can also be calculated as a quadratic Hilbert symbol
	\[ [A,B,C]_{F, \infty} = (\beta, \pi)_{\p_{\infty}} \in \{\pm1\}\]
	where $\pi$ is a uniformiser at $\p_{\infty}$. Here we may also have to apply a twist $\beta \to Q\beta$.
\end{remark}
Before moving to the reciprocity law, we give a property of the R\'edei symbol.
\begin{proposition}\label{prop: rivial special symbol}
Suppose that $D \in \F_q[x]$ is a squarefree polynomial with at least one irreducible component of odd degree. Let $D = D_1D_2$ be a decomposition of the second type as in \cref{lem: decomposition 2nd type ffields}. Then the R\'edei symbol $[D_1, D_2, -D_1D_2]$ is defined and trivial. 
\end{proposition}
\begin{proof}
	The fact that $D_1D_2$ is a decomposition of the second type implies that $(D_1, D_2)_P=1$ for all primes $P$ by \cref{item 3} of \cref{lem: decomposition 2nd type ffields}. Indeed, for finite primes $P\nmid D$ we have $(D_1, D_2)_P = 1$. When $P$ divides $D$, say $D_1$, it holds that $(D_1, D_2)_P = \Leg{D_2}{P}=1$. The vanishing of the symbol at infinity follows from the product formula \cite[theorem $X.ii, 4.13$]{Artin-Tate}. Multiplying $(D_1, D_2)_P$ with the trivial symbols $(D_1, -D_1)_P$ and $(-D_2, D_2)_P$ implies $(D_1, -D_1D_2)_P= (-D_1D_2, D_2)_P=1$ so that the triple satisfies \cref{redei cond 1}. Condition \ref{redei cond 2} is also satisfied because at most one of $D_1$ and $D_2$ is of odd degree and $D$ is a squarefree polynomial. Thus $[D_1, D_2, -D_1D_2]$ is a well-defined R\'edei symbol. 
	
	Since $D_1$ and $D_2$ are relatively prime and $D$ contains an irreducible component of odd degree, a minimally ramified extension $F$ with respect to $-D_1D_2$ is unramified over $K = k(\sqrt{D})$. We can even take $F$ to lie in the Hilbert class field by \cref{lem: decomposition 2nd type ffields}. Consider the ideal $\cc = (\sqrt{D})$ of norm $-D$. This ideal is trivial in the class group $C(D)$ of $K$. Because $F$ is a subextension of the Hilbert class field of $K$, we use Artin reciprocity to find that $\cc$ has trivial Artin symbol in the extension $F/K$. Since the infinite primes split completely in $F/K$, the Artin symbol at infinity is also trivial.

\end{proof}

\subsection{R\'edei reciprocity}\label{subsec: Redei reciprocity}
In this subsection we prove the main result about R\'edei symbols, which by definition is symmetric in its first two arguments. We will show that the symbol is symmetric in all three of its arguments by proving the following theorem.

\begin{theorem}\label{thm: Redei reciprocity}
	For $A,B,C \in k^*/k^{*2}$ satisfying the conditions of \cref{def: Redei symbol}, the R\'edei symbol is multiplicative in each of its arguments and satisfies
	\[[A,B,C] = [B,A,C] = [A,C,B].\]
	We call this the \emph{R\'edei reciprocity law}.
\end{theorem}

In order to prove R\'edei reciprocity, let us give an overview of all the field extensions. Choose a minimally ramified extension $F = E(\sqrt{\beta})$ over $E$ of $K = k(\sqrt{AB})$ with respect to $C$. Similarly, choose a minimally ramified extension $F' = E'(\sqrt{\gamma})$ over $E'$ of $K' = k(\sqrt{AC})$ with respect to $B$. We obtain local symbols $[A,B,C]_{F, P}$ and $[A,C,B]_{F', P}$ for all primes $P$. The elements $\beta, \gamma \in K_A$ have norm $B,C\mod{k^{*^2}}$ respectively. The diagram of extensions can be seen in \cref{fig: diagram}.
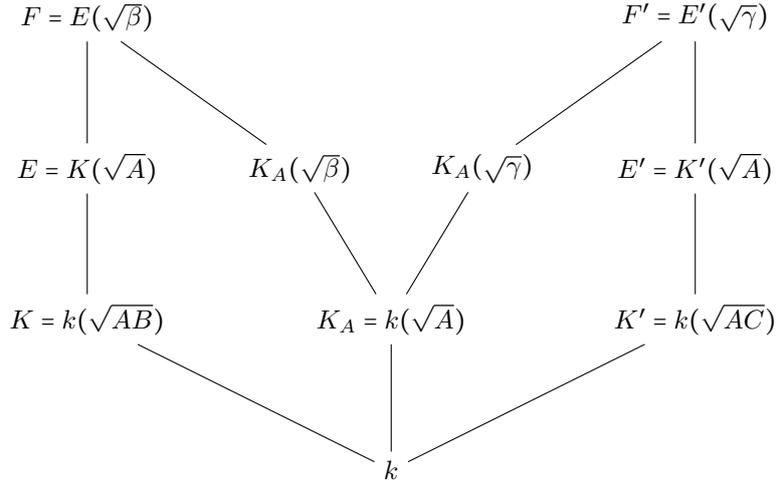
\begin{figure}[h]
\begin{center}
	\begin{tikzpicture}
		\node (Q1) at (0,-1) {$k$};
		\node (Q2) at (-4, 1) {$K = k(\sqrt{AB})$};
		\node (Q3) at (0,1) {$K_A = k(\sqrt{A})$};
		\node (Q4) at (4,1) {$K' = k(\sqrt{AC})$};
		\node (Q5) at (-4,3) {$E = K(\sqrt{A})$};
		\node (Q6) at (4,3) {$E' = K'(\sqrt{A})$};
		\node (Q7) at (-4,5) {$F = E(\sqrt{\beta})$};
		\node (Q8) at (4,5) {$F' = E'(\sqrt{\gamma})$};
		\node (Q9) at (-1.2, 3) {$K_A(\sqrt{\beta})$};
		\node (Q10) at (1.2,3) {$K_A(\sqrt{\gamma})$};
		
		\draw (Q1)--(Q2);
		\draw (Q1)--(Q3);
		\draw (Q1)--(Q4);
		\draw (Q2)--(Q5);
		\draw (Q4)--(Q6);
		\draw (Q5)--(Q7);
		\draw (Q6)--(Q8);
		\draw (Q3)--(Q9);
		\draw (Q3)--(Q10);
		\draw (Q9)--(Q7);
		\draw (Q10)--(Q8);
	\end{tikzpicture}
\end{center}
\caption{: The field extensions used in \cref{lem: local reciprocity}.}{\label{fig: diagram}}
\end{figure}
\begin{lemma}\label{lem: local reciprocity}
	Let $A,B,C \in k^*/k^{*^2}$ satisfy the conditions of \cref{def: Redei symbol} and let $F = E(\sqrt{\beta})$ minimally ramified over $E$ with respect to $C$ and $F' = E'(\sqrt{\gamma})$ minimally ramified over $E'$ with respect to $B$. For all primes $P$ of $k$ it holds that
	\begin{equation}\label{eq: local reciprocity}
		[A,B,C]_{F,P}[A,C,B]_{F',P} = \prod_{\p \mid P \text{ in } K_A} (\beta, \gamma)_\p.
	\end{equation}
\end{lemma}
\begin{proof}
	Denote the left- and right-hand side of \cref{eq: local reciprocity} by $L_P$ and $R_P$ respectively. Note that both sides of \cref{eq: local reciprocity} are symmetric in $B$ and $C$. We may also replace $\beta$ (or $\gamma$) in $R_P$ with its conjugate, because the product equals
	\[\prod_{\p\mid P} (\beta, \gamma)_{\p}\prod_{\p\mid P}(\beta', \gamma)_{\p} = \prod_{\p\mid P} (B, \gamma)_{\p} = (B,C)_P = 1.\]
	
	At most one of $F$ and $F'$ can require an extra twist $\beta \mapsto Q\beta$. Without loss of generality we may assume the twist, if it occurs, happens for $F$. Write therefore $\beta = Q\beta_0$ where $Q$ is either trivial when no twist is required or it is an irreducible polynomial satisfying \cref{eq: extra ramification requirement}. Let $P$ be a prime not equal to $Q$. By condition \ref{redei cond 2} $P$ divides at most $2$ of $A,B,C$. If $P$ divides $B$ it is either split or ramified in $K_A/k$. In the ramified case  $\beta_0$ is, up to squares, a uniformiser at $\p_1\mid  P$ in $K_A$. The twist with $Q$ does not matter because \cref{eq: extra ramification requirement} implies that $Q$ is a unit at $\p_1$. In the split case, $(P) = \p_1\p_2$ in $K_A$, $\beta$ is a unit at the other prime $\p_2$. When $P$ does not divide $B$, the fact that $F/K$ is minimally ramified implies that $\beta$ is a unit at primes $\p\mid  P$ in $K$, up to squares. Analogous statements hold for $C$ and $\gamma$.
	
	We are now ready to check that equality in \cref{eq: local reciprocity} holds for any prime $P$ not equal to $Q$.
	\begin{itemize}
		\item  Suppose first that $P$ does not divide $BC$. Then $L_P=1$ by definition of the symbol. The elements $\beta, \gamma$ are both units at all $\p\mid  P$, so that the symbols $(\beta, \gamma)_{\p}$ are all trivial, making $R_P =1$. 
		\item Next assume that $P$ divides exactly one of $B$ and $C$, say without loss of generality $P$ divides $C$. Then $\beta$ is a is a $P$-unit and $F$ is unramified over $P$, while $\gamma$ is a uniformiser at a prime $\p_1\mid P$ in $K_A$ up to squares. Thus we can apply \cref{eq: redei as hilbert symbol} and write $L_P = (\beta, \gamma)_{\p_1}$. In the case that $P$ ramifies in $K_A$ we have equality to $R_P$. In the split case both $\beta$ and $\gamma$ are units at $\p_2\mid  P$, giving a Hilbert symbol $(\beta, \gamma)_{\p_2}=1$ and the right hand side also equals $(\beta, \gamma)_{\p_1}$.
		\item Finally, when $P$ divides both $B$ and $C$, it cannot divide $A$ and we are in the split case in $K_A$. After replacing $\beta$ by its conjugate $\beta'$ if necessary, we may assume that $\beta$ is a unit at $\p_1$ and a uniformiser at $\p_2$, up to squares, while $\gamma$ is a uniformiser at $\p_1$ and a unit at $\p_2$. By \cref{eq: redei as hilbert symbol} we find 
		\[L_P = [A,B,C]_{F,P}[A,C,B]_{F',P} = (\beta, \gamma)_{\p_1} (\beta, \gamma)_{\p_2} = R_P.\]
	\end{itemize}
	When $Q$ is non-trivial, the local reciprocity in \cref{eq: local reciprocity} still holds at $Q$. As $Q$ does not divide $BC$ we get $L_Q = 1$ by definition. By \cref{eq: extra ramification requirement} $Q$ splits in $K_A$ and both $\beta_0$ and $\gamma$ are units at primes $\q\mid Q$ in $K_A$. Thus $R_Q$ equals
	\[ \prod_{\q \mid Q \text{ in } K_A} (\beta, \gamma)_\q =  \prod_{\q \mid Q \text{ in } K_A} (Q, \gamma)_\q(\beta_0, \gamma)_\q = (Q,C)_Q = 1 = L_Q.\]
\end{proof}

\paragraph*{Proof of \Cref{thm: Redei reciprocity}}
By \cref{lem: local reciprocity} the product of R\'edei symbols $[A,B,C][A,C,B]$ is the product local symbols as in \cref{eq: local artin symbol} over all primes $\p\leq \infty$ in $K_A$. We write
\[[A,B,C][A,C,B] = \prod_{P \leq \infty \text{ in }\F_q(x)} [A,B,C]_{F,P}[A,C,B]_{F',P} = \prod_{\p\leq \infty \text{ in } K_A} (\beta, \gamma)_{\p} = 1.\]
The last equality follows from the product formula. Hence we have symmetry of the R\'edei symbol in the last two arguments. By construction we also have symmetry in the first two arguments. The multiplicativity of the R\'edei symbol in $C$ follows from its construction as a product of all its $P$-parts. Multiplicativity in all its arguments then follows from the symmetry. \qed\linebreak

The reciprocity law gives us a faster way of showing that \cref{def: Redei symbol} is well-defined.
\begin{lemma}
	The definition of R\'edei symbols in \cref{def: Redei symbol} is independent of the choice of minimally unramified $F_{A,B}$ and ideal $\cc$.
\end{lemma}
\begin{proof}
	We already saw the independence of $\cc$ by writing $[A,B,C]$ as a product of $P$-symbols. The fact that it does not depend on the choice of $F_{A,B}$ follows from the reciprocity law $[A,B,C] = [A,C,B]$, because the right hand side is independent of $F_{A,B}$. This also works when $C=B$ by using that $[A,B,B] = [B,B,A]$.
\end{proof}

\section{Governing fields}\label{sec: governing fields}
As an application of R\'edei reciprocity, we will show the existence of a governing field for the $8$-rank whenever $D$ has an irreducible component of odd degree. This assumption is needed to avoid the case that there exists an extra irregular ambiguous ideal class. We recall the definition of a governing field from \cref{def: gov field intro}.

\begin{definition}[Governing fields]
	Let $q$ an odd prime and $D\in \F_q[x]$ squarefree. A Galois extension $\Omega_{2^j}(D)/\F_q(x)$ is called a \emph{governing field for the $2^j$-rank} of $C(DP)$ if the Frobenius conjugacy class of an unramified prime $P\in \F_q[x]$ in $\Gal(\Omega_{2^j}(D)/\F_q(x))$ completely determines the $2^k$-rank for all $k\leq j$.
\end{definition}

\begin{remark}
	In the original formulation for number fields, Cohn and Lagarias showed that if there exists a governing field for the $2^j$-rank, there is a unique one with smallest degree \cite[theorem 1.1]{Cohn-Lagarias}. The proof of this theorem also holds for function fields. We have not required $\Omega_j(D)$ to have the smallest degree, so we are considering \textit{a} governing field instead of \textit{the} governing field.
\end{remark}

Let $D$ be a squarefree polynomial. The governing field $\Omega_{2^0}(D) = \F_q(x)$ exists trivially. By \cref{thm: 2-rank} the $2$-rank of $C(DP)$, with $P$ an irreducible monic polynomial, depends on the leading coefficient of $D$, the parity of the degrees of the irreducible components of $D$, and the parity of $\deg(P)$. A governing field for the $2$-rank therefore only needs to govern the parity of $\deg(P)$. Since the splitting behaviour of $P$ in the extension $\F_q(x) = k \subset k(\sqrt{g})$ is determined by the quadratic symbol $\Leg{g}{P} = (-1)^{\deg(P)}$, we find that
\[\Omega_2(D) = k( \sqrt{g})\cong \F_{q^2}(x)\]
is a governing field for the $2$-rank of $C(DP)$.

By the R\'edei map from \cref{eq: Redei map}, the $4$-rank of $C(DP)$ depends on the Artin symbols $\Art{\a_j}{k( \sqrt{A_i})}{k}$ where the $\a_j$ generate the ambiguous ideal classes and the $\sqrt{A_i}$ generate the genus field $H_2$ in \cref{eq: genus field}. Recall from \cref{eq: R4 entries} that these Artin symbols are determined by genus characters. In particular it holds that
\[\Art{\a_j}{k( \sqrt{A_i})}{k} = \chi_{A_i}(N(\a_j)) = \Leg{A_i}{N(\a_j)},\]
whenever $A_i$ and $N(\a_j)$ are relatively prime.

Most of the generating ambiguous ideals $\a_j$ have a norm that is an irreducible polynomial $P_i\mid D$, except for the potential irregular ambiguous ideal class. In that case we can write $D=c(U^2-gV^2)$ with $\deg(U)>\deg(V)$ \cite{Artin}, so that the irregular class is generated by an ideal $\alpha$ with $N(\alpha) = U$ (or $V$). Note that $N(\alpha)$ is a nonsquare modulo every prime divisor of $D$, so the symbols $\Leg{P_i}{N(\alpha)}$ are fixed. The only symbol that might vary in the column of the R\'edei matrix for the ideal $\alpha$ is the quadratic symbol $\Leg{g}{N(\alpha)}$. For a fixed $D$, we have found that the $4$-rank of $C(DP)$ will be determined by
\begin{itemize}
	\item $\Leg{g}{P} \text{ and } \Leg{P_i}{P} \text{ for each } P_i\mid D \text{ monic irreducible}$, and
	\item $\Leg{g}{N(\alpha)} = (-1)^{\deg(N(\alpha))} = (-1)^{\deg(DP)/2}$ when $\sgn(D)=g$ and $DP$ consists of only irreducible parts of even degree.
\end{itemize}
The symbol $\Leg{g}{N(\alpha)}$, if applicable, is determined by the degree of $P$ modulo $4$. Since $2$-divisibility of $\deg(P)$ is governed by the splitting behaviour in $k(\sqrt{g})\cong \F_{q^2}(x)$, the $4$-divisibility will be governed a further splitting in the constant field extension $k( \sqrt[4]{g}) \cong \F_{q^4}(x)$.  We can conclude that a governing field for the $4$-rank of $C(DP)$ is given by
\[\Omega_4(D) = \begin{cases}
	k( \sqrt[4]{g}, \{\sqrt{P_i}: P_i\mid D \text{ monic irreducible}\}) &\text{ if }  \sgn(D) \text{ and each irreducible}\\
	&\qquad\text{part of $D$ is of even degree,}\\
	k(\sqrt{g}, \{\sqrt{P_i}: P_i\mid D \text{ monic irreducible}\}) &\text{ otherwise.} \\
	
\end{cases}\]
\begin{remark}
	The governing field that we have written above may not be the smallest governing field of $C(D)$.
\end{remark}

\begin{theorem}\label{thm: gov field 8rank ffields}
	Let $D = eP_1\dots P_s \in \F_q[x]$ a squarefree polynomial with all $P_i$ monic irreducible polynomials and assume at least one of the $P_i$ is of odd degree. Then $\Omega_8(D)$ exists.
\end{theorem}
\begin{proof}
	We have already seen that $\Omega_4(D) = \F_q(x)(\sqrt{g}, \{\sqrt{P_i}: P_i\mid D\})$ is a governing field for the $4$-rank of $C(DP)$. Denote it by $\Omega_{4,D}$. Now suppose that $P$ and $P'$ are primes that are unramified in $\Omega_4(D)$ and have the same Artin symbol in $\Gal(\Omega_4(D)/\F_q(x))$. If we number the primes in $DP$ and $DP'$ in the obvious compatible way, the R\'edei matrices $R_4$ and $R_4'$ as given in \cref{eq: Redei map} will coincide. To calculate the $8$-rank via \cref{eq: 8-rank map}, we can choose compatible bases and compare $R_8$ and $R_8'$ by each entry. 
	
	An entry in $R_8$ is determined the Artin symbol $\Art{\m}{E(\beta)}{K} \in \Gal(F/E)$ of some ideal  $\m$ of norm $M$ and a quartic extension $F=E(\sqrt{\beta})/K$ belonging to a splitting $DP = D_1D_2$ of the second type. We would like to write this entry of $R_8$ as a R\'edei symbol $[D_1, D_2, M]$. 
	As we excluded the possibility of an irregular ambiguous ideal, we can take $M$ to divide $DP$. At most one of $D_1$ and $D_2$ is of odd degree, because $DP=D_1D_2$ is a splitting of the second type. Hence the triple $(D_1, D_2, M)$ will satisfy condition \ref{redei cond 1}. By \cref{lem: decomposition 2nd type ffields} all Hilbert symbols $(D_1, D_2)_P$ are trivial. The fact that $\m$ is a generator of $\ker(R_4)$ implies that it is a square in the class group. Thus there exists an ideal $\mathfrak{I}$ in $k(\sqrt{D})$ of some norm $I \in k^*$ such that $\mathfrak{I}^2 \m = (Z)$ is principal with some generator $Z = A+B\sqrt{D}\in O_{k(\sqrt{D})}$. Taking norms of these ideals in $k$, it follows that $I^2M = A^2-DB^2$ which means that $M$ is a norm in $k(\sqrt{D})$. In particular the Hilbert symbols $(D, M)_{P} = (D_1, M)_P(D_2, M)_P$ are trivial for all $P \leq \infty$, giving the terms on the right hand side the same value. We still need to show that both are trivial.
	
	Consider a finite prime $P$. Whenever $P\nmid M$, we can take the $D_i$ that is not divisible by $P$ (because $\gcd(D_1, D_2)=1$), making them both $P$-units and $(D_i, M)_P=1$. When $P$ divides $M$, $P$ must split either in $k_{D_1}$ or $k_{D_2}$ because it is a decomposition of the second type, implying that at least one of $(D_1, M)_P$ and $(D_2, M)_P$ is $1$. Hence $(D_1, M)_P = (D_2,M)_P=1$ for all finite primes $P$. The Hilbert symbol at infinity follows from the product formula \cite[theorem XII.4.13]{Artin-Tate}. We conclude that $[D_1, D_2, M]$ is a well-defined R\'edei symbol. The same reasoning holds for the entries of $R_8'$.
	
	By possibly switching $D_1$ and $D_2$, we may assume that $P$ divides $D_2$. When $P$ also divides $M$ we can use the trivial symbol $[D_1, D_2, -D_1D_2]$ from \cref{prop: rivial special symbol} and additivity of the R\'edei symbol to write
	\[[D_1, D_2, M] = [D_1, D_2, 1/M] = [D_1, D_2, -D_1D_2/M],\]
	where $P$ does not divide $-D_1D_2/M$. All entries in $R_8$ can thus be written as $[D_1, D_2, M]$ where $P\nmid D_1M$. The entries of $R_8'$ have become $[D_1', D_2',M']=[D_1, D_2', M]$, where $P'\nmid D_1M$.
	
	After this reduction we can apply \cref{thm: Redei reciprocity} and rewrite the entries of $R_8$ as
	\[[D_1, D_2, M] = [D_1, M, D_2].\]
	
	The triple $(D_1,M,D_2/P)$ still satisfies the conditions of \cref{def minimal ramification over C}, so there exists an extension $F_{D_1,M}$ that is minimally ramified over $E_{D_1, M} = k(\sqrt{M}, \sqrt{D_1})$ with respect to $D_2/P$. In fact, we claim that the extension $F_{D_1, M}$ will always be in case 1 of \cref{def minimal ramification over C}.	
	\begin{claim}\label{claim: no exception case}
		The extension $E_{D_1,M} \subset F_{D_1,M}$ is unramified outside $\gcd_{\infty}(D_1, M).$
	\end{claim}
	Writing $F_{D_1, M} = E_{D_1, M}(\beta_{D_1, M})$, the result is that $[D_1, M, D_2]=[D_1, M, D_2']$ whenever $P$ and $P'$ have the same Artin symbol in $\Gal(\Omega_{4,d}(\sqrt{\beta_{D_1, M}}, \sqrt{Q})/\Omega_{4,D})$. Letting $\tilde{\Omega_8}(D)$ be the compositum of the fields $\Omega_{4,D}(\sqrt{\beta_{D_1, M}})$ ranging over all the entries of $R_8$ gives us the desired governing field.
\end{proof}

\begin{proof}[Proof of \cref{claim: no exception case}]
	For the triple $(D_1, M, D_2)$ we want to show that there exists a minimal ramified extension $F_{D_1, M}$ over $k(\sqrt{D_1}, \sqrt{M})$ with respect to $D_2$ that satisfies \cref{good case} of \cref{def minimal ramification over C}, i.e. unramified outside $\gcd_{\infty}(D_1, M).$ Let us recall what $D_1$ and $M$ actually can be. In the construction of the matrix $R_8$ we are free to choose a basis of $\ker(R_4)$ (representing classes in $C^2\cap C[2]$) and a basis of $\hat{C}[4]/\hat{C}[2]$ (which can be represented by quadratic characters via \cref{lem: decomposition 2nd type ffields}). We can therefore use an explicit description of the basis of $C[2]$ to narrow down the possibilities for $M$. Let $D = eQ_1\dots Q_{s_1}P_1\dots P_{s_2}$ be the decomposition of $D$ into irreducible parts with $s_1>0$ by assumption. In \cref{subsec: 2rank} we found that $C[2]$ can be generated by the classes of all finite primes dividing $DP$. We can therefore take $M$ to be in 
	\begin{equation}\label{eq: M options}
		M \in \{Q_1,\dots,Q_{s_1}, P_1 ,\dots, P_{s_2}\}\cup \{-D\},
	\end{equation}
	By making the changes $P_j \mapsto Q_1P_j$ for all $j$ and the extra change $P\mapsto Q_1P$ when $\deg(D)$ is even to the basis of $C[2]$, all the options for $M$ will have odd degree. In other words we may assume that 
	\begin{equation}\label{eq: M options}
		M \in \{Q_1,\dots,Q_{s_1},Q_1 P_1 ,\dots, Q_1P_{s_2}\}\cup \{-D \text{ or } -D/Q_1\}.
	\end{equation}
	 That implies the pair $(D,M)$ always lands in case 1 of \cref{def minimal ramification over C}.
\end{proof}

\paragraph{Further research}
We have shown that \cref{bad case} (the `bad case') of minimally ramified extensions as in \cref{def minimal ramification over C} cannot be avoided on certain occasions. A next step would be to give necessary and sufficient conditions when a triple $(A,B,C)$ lands in the bad case.

Furthermore, one may wonder whether governing fields for the $8$-rank exist when all irreducibe part of the discriminant have even degree. Then the $8$-rank is much harder to control by the existence of an irregular ambiguous ideal. We might compare this case to the 16-rank of certain quadratic number fields, as described by Milovic \cite{Milovic}. In both situations the discriminant can be written as a difference $D = A^2 -gB^2$ for some unit $g$. For number fields, Koymans and Milovic \cite{Koymans-Milovic} showed that governing fields for the 16-rank do not exist for such discriminants and their methods may be applied to function fields as well.
 
\paragraph{Funding} The author is supported by the Dutch Research Council (NWO) as part of the Quantum Software Consortium programme (project number 024.003.037) and the Quantum Delta NL programme (project number NGF.1623.23.023).
 
\paragraph{Acknowledgements}
Part of the results were obtained in the author's Master's thesis \cite{Stokvis}. I want to thank Gunther Cornelissen for his supervision and helpful suggestions during the project. I would also like to thank Peter Koymans and Jonathan Love for the useful discussions, and Peter Bruin for proofreading the document and useful suggestions.

\bibliographystyle{alphaurl}
\bibliography{bibfile}
\thispagestyle{plain}

\appendix

\section{Magma}\label{app: magma}
The following Magma code generates examples of minimally ramified extensions where the extra ramified prime from \cref{def: minimal ramification} is necessary. It generates one example for each prime $q\equiv 3\mod{4}$ smaller than $200$.
	
\begin{lstlisting}[breaklines, tabsize=4, basicstyle=\footnotesize]
//Input: q = field size, f and g make up the decomposition of the second type, 
//	C is a point on the conic (X^2=f*Y^2+g*Z^2).
//Output: Decomposition type of an infinite place in the extension F_q(t, sqrt(f),
//	sqrt(X-sqrt(f)*Y)) / F_q(t, sqrt(f)). We want this to be unramified.
ramification := function(q,f,g,C)
	R<t> := FunctionField(GF(q));
	P<y> := PolynomialRing(R);
	X:= Numerator(C[1]);
	Y:= Numerator(C[2]);
	FF1<alpha> := FunctionField(y^2-f);
	P := Places(R,1);
	inf:=P[1];
	
	//The infinite prime should split in F_q(t, sqrt(f)) / F_q(t)
	DecompositionType(FF1, inf);
	
	Q<z> :=PolynomialRing(FF1);
	FF2<beta>:= FunctionField(z^2 - (X+alpha*Y));
	P := Places(FF1, 1);
	inf:=P[1];
	return (DecompositionType(FF2, inf));
end function;
	
for q in [3..200] do
	if q mod 4 eq 3 and IsPrimePower(q) then 
		q;
		R<t>:= RationalFunctionField(GF(q));
		S<u>:=PolynomialRing(R);
		P<x,y,z> := ProjectiveSpace(R, 2);
		F<t>:=PolynomialRing(GF(q));
		f:= RandomIrreduciblePolynomial(GF(q),4);
		for g in AllIrreduciblePolynomials(GF(q), 2) do 
		C := Conic(P, x^2 -f*y^2 - g*z^2);
		if HasRationalPoint(C) then
			RP:=RationalPoint(C);
			if Degree(Numerator(RP[1])) mod 2 eq 1 then
				f,g;
				RP;
				ramification(q,f,g, RP);
				clgroup(q,f,g);
				break;
			end if;
		end if;
		end for;
	end if;
end for;
\end{lstlisting}
\end{document}